\documentclass[a4paper,11pt,envcountsame]{llncs}
\usepackage[hmargin=4cm,tmargin=42mm,bmargin=48mm]{geometry}
\usepackage{amsfonts,amsmath,amssymb,enumerate}
\usepackage[hidelinks]{hyperref}
\usepackage[all]{xy}

\newcommand{\on}{\operatorname}
\newcommand{\F}{\mathbb F}
\newcommand{\N}{\mathbb N}
\newcommand{\Z}{\mathbb Z}
\newcommand{\Q}{\mathbb Q}
\newcommand{\smprod}{\textstyle\prod\limits}
\newcommand{\set}[1]{\{ #1 \}}
\newcommand{\ev}{\on{ev}}
\newcommand{\ol}{\overline}

\newcommand{\Res}{\on{Res}}
\newcommand{\fa}{\mathfrak{a}}
\newcommand{\fp}{\mathfrak{p}}
\newcommand{\mcO}{\mathcal{O}}
\newcommand{\eps}{\varepsilon}
\renewcommand{\phi}{\varphi}

\title{Indiscreet logarithms\\ in finite fields of small characteristic}

\author{Robert Granger\inst{1}%
  \thanks{Supported by the Swiss National Science Foundation via grant
    number 200021-156420.}
  \and Thorsten Kleinjung\inst{2}%
  \and Jens Zumbr\"agel\inst{1}%
}

\institute{Laboratory for Cryptologic Algorithms\\
  \'Ecole polytechnique f\'ed\'erale de Lausanne, Switzerland
  \and
  Institute of Mathematics, Universit\"at Leipzig, Germany\\
  \email{\{robert.granger,thorsten.kleinjung,jens.zumbragel\}@epfl.ch}}

\begin{document}

\maketitle

\begin{abstract}
  Recently, several striking advances have taken place regarding the
  discrete logarithm problem (DLP) in finite fields of small
  characteristic, despite progress having remained essentially static
  for nearly thirty years, with the best known algorithms being of
  subexponential complexity.  In this expository article we describe
  the key insights and constructions which culminated in two
  independent quasi-polynomial algorithms.  To put these developments
  into both a historical and a mathematical context, as well as to
  provide a comparison with the cases of so-called large and medium
  characteristic fields, we give an overview of the state-of-the-art
  algorithms for computing discrete logarithms in all finite fields.
  Our presentation aims to guide the reader through the algorithms and
  their complexity analyses \emph{ab initio}.

  \keywords{discrete logarithm problem, finite fields,
    quasi-polynomial algorithms}
\end{abstract}


\setcounter{footnote}{0}


\section{Introduction}

Given a finite cyclic group $(G, \cdot)$, a generator $g \in G$ and
another group element $h \in G$, the discrete logarithm problem (DLP)
is the computational problem of finding an integer~$x$ such that $g^x
= h$.  The integer~$x$ -- denoted by $\log_g h$ -- is uniquely
determined modulo the group order and is called the \emph{discrete
  logarithm} of~$h$ to the base~$g$.  The most natural example is the
multiplicative group $\F_p^{\times}$ of the field $\F_p$ of integers
modulo a prime~$p$, which is a cyclic group in which the DLP is
believed to be hard in general.  Besides this classical case, several
other groups have been extensively studied, including the
multiplicative group $\F_q^{\times}$ of any finite field $\F_q$ of
prime power order $q = p^n$, the group $E(\F_q)$ of $\F_q$-rational
points on an elliptic curve~$E$, or the group~$J_C(\F_q)$ associated
with the Jacobian $J_C$ of a (usually hyperelliptic) curve~$C$ of
genus~$>1$.

The study of discrete logarithms can be traced back to at least two
centuries ago, when they appeared in Gau\ss' \emph{Disquisitiones
  Arithmeticae} and were referred to as \emph{indices}
\cite[art. 57--60]{Gauss}.  Beyond its relevance as a natural
computational problem, the study of the DLP really came of age with
the advent of public-key cryptography in~1976, as the hardness of the
DLP (originally in~$\F_p^{\times}$) forms the basis of the famous
Diffie-Hellman protocol~\cite{DH76} and other cryptographic
primitives.

To describe the hardness of the DLP in a group of order~$N$ one
usually considers the asymptotic complexity in the input size of the
problem, which is proportional to $\log N$.  To indicate the
complexity, it has become customary to use the notation
\[ L_N(\alpha, c) \,=\, \exp \big( (c+o(1)) (\log N)^{\alpha} 
(\log \log N)^{1-\alpha} \big) \,, \]
where $\alpha \in [0,1]$, $c > 0$ and $\log$ denotes the natural
logarithm; we may omit the subscript~$N$ when there is no ambiguity
and sometimes write $L(\alpha)$ to mean $L(\alpha, c)$ for some $c >
0$.  Note that $L(0) = (\log N)^{c + o(1)}$, which corresponds to
polynomial time, while $L(1) = N^{c + o(1)}$ corresponds to
exponential time.  An algorithm with a running time of $L(\alpha)$ for
some $0 < \alpha < 1$ is said to be of \emph{subexponential}
complexity.

Algorithms for solving the DLP can be broadly classified into two
classes.  One class consists of the so-called generic algorithms,
which do not exploit any representational properties of group elements
and may thus be applied to any group.  Generic algorithms, like
Pollard's rho method~\cite{Pol78}, necessarily have a running time of
$\Omega( \sqrt N)$~\cite{Nec94,Sho97b}, if~$N$ is prime; if~$N$ is
composite then one may reduce the original problem to the DLP in prime
order subgroups, using the Pohlig-Hellman algorithm~\cite{PH78}.  The
other class consists of the so-called index calculus algorithms, which
exploit a suitable representation of group elements, which imbues a
notion of smoothness.  A basic version of index calculus was analysed
by Adleman~\cite{Adl79} and shown to have complexity~$L(1/2)$.
Subsequently, more advanced index calculus algorithms with lower
complexity have been developed.

In particular, the first $L(1/3)$ algorithm for computing discrete
logarithms, published in 1984 by Coppersmith~\cite{Cop84}, targeted
binary finite fields.  Later, the number field sieve, originally
devised for the integer factoring problem~\cite{LL93}, was adapted for
the DLP in prime fields~\cite{Gor93,Sch93} and resulted again in an
$L(1/3)$ algorithm.  Inspired by the number field sieve, the function
field sieve was developed for computing discrete logarithms in small
(fixed) characteristic~\cite{Adl94,AH99,JL02,JL06}.  Regarding general
finite fields~$\F_{p^n}$, for $n \ll (\log p)^{1/2}$ the number field
sieve and for $\log p \ll n^{1/2}$ the function field sieve still have
complexity $L(1/3)$.  Finally, in 2006 the number field sieve was
generalised~\cite{JLSV06} to work also in the remaining cases, thereby
obtaining $L(1/3)$ algorithms for all families of finite field DLPs.
Some improvements of the medium number field sieve have been reported
recently, see e.g., \cite{BGGM15,BGK15}.

Regarding the case of small characteristic, dramatic advances have
recently taken place.  Indeed, nearly thirty years after Coppersmith's
algorithm the $L(1/3)$ barrier was broken in a series of remarkable
results, starting with the work of G\"olo\u{g}lu, Granger, McGuire and
Zumbr\"agel~\cite{GGMZ13}, and independently Joux~\cite{Jou14}.  The
approach of Joux led to the first heuristic quasi-polynomial
algorithm, i.e., one with running time $\exp( O((\log \log N)^2) ) =
(\log N)^{O(\log \log N)}$, which is in $L(o(1))$, due to Barbulescu,
Gaudry, Joux and Thom\'e~\cite{BGJT14}.  Independently, the approach
of G\"olo\u{g}lu \emph{et al}.\ was later developed into an
alternative algorithm by Granger, Kleinjung and Zumbr\"agel, which
also has quasi-polynomial complexity, but is rigorous for a large
family of fields~\cite{GKZ14b,GKZ15}.

In this article we describe the key insights and constructions
underlying the recent algorithms.  To put these developments into both
a historical and a mathematical context, as well as to provide a
comparison with the cases of so-called large and medium characteristic
fields, we give an overview of the state-of-the-art algorithms for
computing discrete logarithms in all finite fields.  All necessary
mathematical prerequisites will be briefly introduced along the~way.

As our focus is on DLP algorithms for finite fields, we cover generic
algorithms only briefly.  We refer to Odlyzko's paper~\cite{Odl00} or
the recent survey~\cite{JOP14} for an overview on these and other
general aspects of the DLP, including curve-based groups, which are
not covered here.

\paragraph{A remark on terminology.}

The cardinality of any finite field~$F$ is a prime power $|F| = p^n$,
where the prime~$p$ is called the \emph{characteristic} of~$F$, which
is the smallest positive integer~$m$ such that $m 1_F = 0$.
Conversely, given any prime power $q = p^n$ there exists a finite
field of size~$q$, which is unique up to isomorphism and is denoted
by~$\F_q$.

When considering families of finite fields of order~$q = p^n$, where
$p = L_q(\alpha)$ and $q \to \infty$, then different DLP algorithms
apply depending on~$\alpha$.  Accordingly, in the case $\alpha > \frac
2 3$ we speak of \emph{large} characteristic; for $\alpha \in (\frac 1
3, \frac 2 3)$ we speak of \emph{medium} characteristic; and for
$\alpha \in (0, \frac 1 3)$ we say that the characteristic is
\emph{medium-small}, while the boundary cases $\alpha = \frac 2 3$ and
$\alpha = \frac 1 3$ are special cases to be treated additionally.
Finally, in the case $\alpha = 0$, i.e., if~$p$ is of polynomial size
in $\log q$, we speak of finite fields of \emph{small} characteristic.
In particular, if $q = p$ is a prime or $q = p^n$ with~$n$ fixed, then
we are in (very) large characteristic $p = L_q(1)$, whereas if the
characteristic~$p$ is fixed, then we have small characteristic $p =
L_q(0)$; note however that $p = L_q(1)$ (or $p = L_q(0)$) does not
imply that~$n$ (or~$p$) has to be fixed.


Furthermore, for complexity considerations we make use of the notation
$f \approx g$ to indicate that $f / g \to 1$ as the argument goes to
infinity.

\paragraph{Outline.}

We discuss the DLP for general finite cyclic groups in
Section~\ref{sec:general} and briefly present the most common
cryptographic applications and generic algorithms.  The index calculus
method serves as a framework for all advanced DLP algorithms for
finite fields and will be described in Section~\ref{sec:icm}, where we
first present it abstractly in a general group and then give some
basic concrete instances.  Section~\ref{sec:nfs} is devoted to the
number field sieve, which is currently the fastest method for DLP in
both large and medium characteristic, while Section~\ref{sec:ffs} deals
with the function field sieve in finite fields of small (and
medium-small) characteristic. Section~\ref{sec:new} presents the
recent dramatic developments for small characteristic finite fields, 
as well as the key insights and ideas which led to them.
Finally, we conclude in Section~\ref{sec:conclusion}.


\section{The DLP in finite cyclic groups}\label{sec:general}

The discrete logarithm problem can be formulated for any finite cyclic
group; indeed, most cryptographic protocols using the DLP can be
formulated in this abstract setting.  We present in this section the
most important cryptographic applications and the common generic
algorithms.  Since our main focus is the DLP in finite fields to which
index calculus algorithms may be applied, we will here be rather brief
in our treatment.

Let $(G, \cdot)$ be a finite cyclic group of order~$N$ and let $g \in
G$ be a generator.  We assume that the group operation can be computed
efficiently, i.e., in polynomial time, and that the group order is
known.  The surjective group homomorphism $\Z \to G$, $x \mapsto g^x$
induces a group isomorphism
\[ \psi \colon \Z / N \Z \to G \,, \quad [x] \mapsto g^x \,, \]
with inverse map $\psi^{-1} = \log_g \colon G \to \Z / N \Z$, $h
\mapsto \log_g h$.  The map $\psi$ can be computed efficiently by
using a square-and-multiply method, whereas the computation of
$\psi^{-1}$ is believed to be a difficult problem in general; this is
of course the~DLP.  Indeed, for so-called \emph{generic algorithms},
in which one can only perform group operations and check equality of
elements, the DLP necessarily has an exponential running time of
$\Omega( \sqrt N)$ if $N$ is prime~\cite{Nec94,Sho97b}.  However, for
particular groups with concrete representations, faster algorithms
exist, and it is not proven in any such cases that the DLP is hard.

One should note that the DLP, as well as the integer factorisation
problem, can be solved in polynomial time using quantum
computers~\cite{Sho97a}, but in this article we confine ourselves to
the classical computational model.

\paragraph{Cryptographic applications.}

The difficulty of the DLP is nowadays widely used, e.g., for secure
communication over the Internet.  Virtually all public-key
cryptosystems in use today are based on either the integer
factorisation problem or the DLP.  Some common cryptographic protocols
using the DLP are briefly presented below.  In each case the group
$(G, \cdot)$ and a generator~$g \in G$ are assumed to be publicly
known.

In the Diffie-Hellman key-agreement protocol~\cite{DH76} two parties,
usually referred to as Alice and Bob, choose random integers~$a$
and~$b$, respectively, and exchange the group elements $g^a$ and $g^b$
over the public channel, hence they both can compute a common session
key $(g^b)^a = g^{ab} = (g^a)^b$.  Clearly, if the DLP in~$G$ is
feasible, then the key can be computed from the public information, so
it is essential that the DLP is hard.

The Diffie-Hellman protocol can be transformed into a public-key
encryption scheme as shown by ElGamal~\cite{ElG85}.  Indeed, if Bob
has announced his public key $g^b$ and Alice has a secret message $m
\in G$ for Bob, she chooses a random integer~$a$ and sends the pair
$(g^a, m (g^b)^a)$ to Bob, who can decrypt by computing $(g^a)^b =
(g^b)^a$.  Moreover, there are digital signature schemes based on the
DLP, e.g., the ElGamal~\cite{ElG85} and Schnorr~\cite{Sch91} signature
schemes.

More recently, the advent of pairing-based cryptography in 2000
provided new cryptographic functionalities, starting with
identity-based non-interactive key distribution~\cite{SOK00},
one-round tripartite key-agreement~\cite{Jou00} and identity-based
encryption~\cite{BF01}, followed by numerous others.  It also renewed
interest in the underlying DLPs and related security assumptions.

To summarise, the hardness of the DLP in any group in which any of the
above protocols is instantiated is essential to their security.  It
is therefore imperative that these DLPs continue to be studied.

\paragraph{Generic algorithms.}

We now present examples of generic algorithms for solving the DLP in
any finite cyclic group.  Suppose that we want to find the discrete
logarithm $\log_g h$ for a target element $h \in G$.  Recall that $N =
|G|$ is the group order.

In the \emph{Baby-Step-Giant-Step} method, attributed to Shanks, 
let $M = \lceil \sqrt N \rceil$.  One computes a table $\set{ (j,
  g^j) \mid j \in \set{ 0 \dots M\!-\!1 } }$ (baby steps) and sorts it
by the second component.  Then one computes $k = g^{-M}$, as well as
$h, h k, h k^2, \dots$ (giant steps) until a collision $h k^i = g^j$
is detected, in which case the output is $\log_g h = i M + j$.  The
method requires $O(\sqrt N)$ database entries and $O(\sqrt N \log N)$
group operations (or $O(\sqrt N)$ if hash tables are used).

\emph{Pollard's rho} method~\cite{Pol78} reduces the memory requirement
to a negligible amount while preserving the square-root running time.
Therefore, it is the preferred method in practice.  However, due to
its randomised nature the analysis is more intricate.  The idea is to
define pseudorandom sequences $(k_i)$ in~$G$ and $(a_i)$, $(b_i)$
in~$\Z / N \Z$ such that $g^{a_i} h^{b_i} = k_i$ holds for any~$i$.
If $k_j = k_{j+{\ell}}$ holds for some $j, \ell > 0$, then there
exists also~$i$ with $k_i = k_{2i}$, and such collisions can be easily
detected.  In this case, $\log_g h = \frac{a_{2i} - a_i} {b_i -
  b_{2i}}$ is found, provided that the denominator is invertible
modulo~$N$.

Moreover, there is the \emph{Pohlig-Hellman} method~\cite{PH78}, the
efficiency of which depends on the prime factorisation of $N =
p_1^{e_1} \cdot \ldots \cdot p_r^{e_r}$, which we assume is given.
Then one has $\Z / N \Z \cong \Z / p_1^{e_1} \Z \times \dots \times \Z
/ p_r^{e_r} \Z$ by the Chinese Remainder Theorem, hence $x = \log_g h$
corresponds to a tuple $(x_1, \dots, x_r)$ and one can consider the
DLP for each factor $x_i \in \Z / p_i^{e_i} \Z$ separately.
Furthermore, if $e_i > 1$, then $x_i \in \Z / p_i^{e_i} \Z$ can be
found by obtaining the ``p-ary'' digits of~$x_i$, one at a time
starting with the least significant digit, by solving a DLP of
order~$p_i$.  Therefore, the Pohlig-Hellman method essentially reduces
the DLP complexity in a group of order~$N$ to a group of order the
largest prime factor of~$N$, for which either the Baby-Step-Giant-Step
or Pollard's rho method may be applied.

We remark that, although we are mainly interested in the DLP in finite
fields -- to which more efficient index calculus algorithms apply --
to compute the DLP modulo the order of the full multiplicative group
it is often preferable to use the Pohlig-Hellman method for computing
the solution modulo the small prime power factors.


\section{The index calculus method}\label{sec:icm}

The index calculus method can be much more efficient than generic
algorithms for certain groups and representations, and is the template
for all subexponential DLP algorithms.  While we describe the
framework for any group, the details will depend on the concrete
representation of the group elements.

As before let $\log_g h$ be the target logarithm in a cyclic group~$G$
of order~$N$.  We choose a subset $F \subseteq G$ such that $\langle F
\rangle = G$ (often we have $g \in F$), called the \emph{factor base}.
The idea is to first obtain $\log_g f$ for all $f \in F$.  Let $A = \Z
/ N \Z$ and consider the surjective group homomorphism
\[ \phi \colon A^F \to G \,, \quad (e_f)_{f \in F} \mapsto 
\smprod_{f \in F} f^{e_f} \,. \]

The index calculus method consists of three phases.
\begin{enumerate}
\item {\bf Relation generation.} Find vectors $(e_f)_{f \in F}$ in
  $\ker \phi$, called \emph{relations}, thus generating a subset $R
  \subseteq \ker \phi$.
\item {\bf Linear algebra.} Compute an element $0 \ne (x_f)_{f \in F}
  \in R^{\perp}$, i.e., satisfying $\sum_{f \in F} x_f e_f = 0$ for
  all $(e_f)_{f \in F} \in R$.
\item {\bf Individual logarithm.} Find a preimage $(e_f)_{f \in F} \in
  \phi^{-1}(h)$, for then we have found $\log_g h = \sum_{f \in F} e_f
  \log_g f$, cf.~Lemma~\ref{lem:relations}.
\end{enumerate}

As the next result shows, provided that enough relations have been
found, the discrete logarithms of the factor base elements are
determined up to a scalar multiple.  In practice, the condition given
in the lemma is satisfied, if a few more relations than factor base
elements have been obtained.

\begin{lemma}\label{lem:relations}
  Suppose that $\on{span} R = \ker \phi$ and $(x_f)_{f \in F} \in
  R^{\perp}$, then there exists $\lambda \in A$ such that $x_f =
  \lambda \log_g f$ for all $f \in F$.
\end{lemma}

\begin{proof}
  It holds that $R^{\perp} = (\on{span} R)^{\perp} = (\ker
  \phi)^{\perp}$ and $A^F / \ker \phi \cong \on{im} \phi \cong A$,
  while $(\ker \phi)^{\perp} \cong A^F / \ker \phi$ as can be seen by
  considering the Smith normal form, hence $R^{\perp} \cong A$.  On
  the other hand, we have $(\log_g f)_{f \in F} \in R^{\perp}$;
  indeed, since $\log_g \colon G \to A$ is a group homomorphism, it
  holds $\sum_{f \in F} e_f \log_g f = \log_g( \prod_{f \in F} f^{e_f}
  ) = 0$ for all $(e_f)_{f \in F} \in R \subseteq \ker \phi$. \qed
\end{proof}

\paragraph{Sparse linear algebra.} 

After the relation generation step an $r \times m$-matrix~$M$ over $\Z
/ N \Z$ has been found, where~$m$ is the size of the factor base~$F$
and $r \approx m$ is the number of relations.  In order to obtain the
factor base logarithms we need a solution $0 \ne x \in (\Z / N \Z)^m$
of $M x = 0$.  Due to the relation generation method the matrix is
usually of low average row weight~$w$.  For such sparse matrices
iterative algorithms are available, most commonly used are the Lanczos
method~\cite{Lan50} (for an adaption to finite fields see, e.g.,
\cite{LO91}) or the Wiedemann algorithm~\cite{Wie86}.  Their cost is
dominated by repeated computations of matrix-vector products $M v$,
and the running time is $O(m^2 w)$ operations in $\Z / N \Z$.
Provided that $\log w = o(\log m)$ (and $\log \log N = o(\log m)$),
which is usually the case, this is of complexity~$m^{2 + o(1)}$.

We note that, as $O(m)$ divisions in $\Z / N \Z$ are necessary (when
using the Lanczos method), the group order~$N$ should avoid small
prime factors, therefore the Pohlig-Hellman reduction and Pollard's
rho algorithm should be used for the small prime power factors.  In
practice, for large-scale computations the linear algebra step poses
some challenges, as the iterative algorithms are not easily
parallelisable.  We also remark that so-called structured Gaussian
elimination (cf.~\cite{LO91,JL03}) can be used to decrease the matrix
dimension~$m$, while increasing the weight~$w$ only moderately.


\subsection{A variant permitting rigorous analysis}%
\label{ssec:egd}

The following variant of the index calculus method, proposed by Enge
and Gaudry~\cite{EG00} and subsequently refined by Diem~\cite{Die11},
is valuable for theoretical analysis.  With this variant one can
compute discrete logarithms, provided only that it is feasible to
express group elements as products over the factor base, as in the
individual logarithm step of the classical index calculus method.

As before, let $(G, \cdot)$ be a cyclic group of order~$N$, let $g \in
G$ be a generator and let $h \in G$ be the target element for the DLP.
Suppose that $F \subseteq G$ is a factor base of cardinality $|F| =
m$.  We choose $a, b \in \Z / N \Z$ uniformly and independently at
random and try to express $g^a h^b$ as a product $g^a h^b = \prod_{f
  \in F} f^{e_f}$.  Once more than~$m$ such expressions have been
found, we consider the matrix consisting of the collected rows
$(e_f)_{f \in F}$ over $\Z / N \Z$, and compute using invertible row
transformations a row echelon form, which contains a vanishing row.
Contrary to the classical index calculus method we do not require a
rank condition for this matrix.  Applying the invertible row
transformations also to the numbers~$a$ and~$b$, then considering the
vanishing row we obtain an identity $g^{a'} h^{b'} = 1_G$.  One can
show that~$b' \in \Z / N \Z$ is uniformly distributed, so that $b'$ is
coprime to~$N$ with high enough probability, in which case $\log_g h =
\frac{a'} {b'} \in \Z / N \Z$ has been found.

Instead of computing a row echelon form by a variant of the Gau\ss\
algorithm one may use sparse linear algebra techniques, which have an
improved running time, however their analysis is more difficult and
the above algorithm has to be modified~\cite{EG00}.  In particular, it
is then necessary to fulfil rank conditions for the generated matrix
following a technique from Pomerance~\cite{Pom87}.


\subsection{Basic concrete versions}

The class of groups to which the index calculus method is applicable
includes the multiplicative groups of prime fields and fields of small
fixed characteristic.  We describe for these cases a simple index
calculus method and provide a running time analysis, which also serves
as a basis for the more advanced index calculus algorithms.

Suppose that~$G$ is $\F_p^{\times}$, the multiplicative group of a
prime field $\F_p$, of order $N = p-1$, with a given generator
$g \in G$.  As factor base we choose
\[ F = \set{ f \mid f \le B,\, f \text{ prime} } \subseteq G \] for
some bound~$B$ (by slight abuse of notation, for $f \in \Z$ we denote
the class $[f] \in \F_p$ also by~$f$).  For simplicity, we assume that
$g \in F$ (otherwise, we include it into the factor base).  To
generate relations, for random $e \in \Z / N \Z$ we compute $g^e \in
\F_p$, lift it to an element in $\N$, and check by trial division
whether it has only prime divisors $\le B$.  If successful, we obtain
a relation $g^e \equiv \prod_{f \in F} f^{e_f} \bmod p$ in~$G$.  Once
enough relations (more than $|F|$) have been found, we compute $\log_g
f$ for all $f \in F$ by solving a linear system over~$\Z / N \Z$.
Finally, given a target element $h \in G$ we similarly obtain one more
relation of the form $h g^e = \prod_{f \in F} f^{e_f}$ to obtain
$\log_g h$.

Considering a finite field $\F_q = \F_{p^n}$ of fixed
characteristic~$p$, note that $\F_{p^n}$ is usually represented as a
quotient ring $\F_p[X] / ( I )$, where $I \in \F_p[X]$ is an
irreducible polynomial of degree~$n$.  For $G = \F_q^{\times}$, it is
then straightforward to adapt the basic index calculus method for
$\F_p^{\times}$ described above to the present situation.  In
particular, as factor base we choose all irreducible polynomials in
$\F_p[X]$ of some bounded degree~$b$, i.e.,
\[ F = \set{ f \mid f \in \F_p[X],\, \deg f \le b,\, f \text{
    irreducible} } \subseteq G \] 
(where we employ a similar abuse of notation).  It suffices in
practice to include only the monic polynomials into the factor base.
In fact, one may perform the discrete logarithm computation in
$\F_q^{\times} / \F_p^{\times}$, i.e., ignoring constants
in~$\F_p^{\times}$, to obtain $\log_g h$ modulo $\frac N {p-1}$.
Using the Pohlig-Hellman method with the fact that $p-1$ divides the
product of the small prime power divisors of the group order $N = p^n
- 1$, the remaining information of $\log_g h$ is deduced easily.


\subsection{Complexity analysis}

A positive integer is called \emph{$B$-smooth} if all its prime
divisors are~$\le B$.  The (heuristic) running time analysis for the
basic index calculus method in $\F_p^{\times}$, as well as for the more
advanced algorithms presented in Section~\ref{sec:nfs}, is based on
the following result on the asymptotic density of smooth numbers among
the integers.

\begin{theorem}[Canfield, Erd\H{o}s, Pomerance~\cite{CEP83}]%
  \label{thm:cop}\sloppy
  A uniformly random integer in $\set{1, \dots, M}$ is $B$-smooth with
  probability \[ P \,=\, u^{-u (1 + o(1))} \,, \quad
  \text{where} \quad u = \tfrac{\log M}{\log B} \,, \]
  provided that $3 \le u \le (1 - \eps) \frac{ \log M } { \log \log M }$
  for some $\eps > 0$.
\end{theorem}

We recall the notation $L_N(\alpha, c) = \exp \big( (c+o(1)) (\log
N)^{\alpha} (\log \log N)^{1-\alpha} \big)$, where $\alpha \in [0,1]$,
$c > 0$, which is originally due to Pomerance~\cite{Pom83} for $\alpha
= \frac 1 2$ and Lenstra and Lenstra~\cite{LL91} for general $\alpha$.

\begin{corollary}\label{cor:cop}
  Let $M = L_N(2 \alpha, \mu)$ and $B = L_N(\alpha, \beta)$, then the
  expected number of trials until a uniformly random number in
  $\set{1, \dots, M}$ is $B$-smooth is $L_N(\alpha, \frac {\alpha \mu}
  {\beta} )$.
\end{corollary}

For analysing the basic version of the index calculus method in
$\smash{\F_p^{\times}}$, we set the smoothness bound $B = L(\frac 1 2,
\beta)$ and we have $M = N = L(1, 1)$.  As we need about $|F| \approx
B / \log B \le B$ relations, our estimated running time equals
\[ L(\tfrac 1 2, \beta) \cdot L(\tfrac 1 2, \tfrac 1 {2 \beta}) =
L(\tfrac 1 2, \beta + \tfrac 1 {2 \beta}) , \] and the optimal choice
\smash{$\beta = \frac 1 {\sqrt 2}$} results in a running time of
$L(\frac 1 2, \sqrt 2)$ for the relation generation.  The linear
algebra running time (using iterative techniques for sparse matrices)
is about $B^2 = L(\frac 1 2, 2 \beta) = L(\frac 1 2, \sqrt 2)$ as
well, while the individual logarithm phase is of lower complexity.

Similarly, a polynomial is called \emph{$b$-smooth} if all its
irreducible factors are of degree~$\le b$; hence, the $1$-smooth
polynomials are precisely those that split into linear factors.

\begin{theorem}[Odlyzko~\cite{Odl85}, Lovorn~\cite{Lov92}]\label{thm:ol}
  A uniformly random polynomial $f \in \F_q[X]$ of degree~$m$ is
  $b$-smooth with probability $P = u^{-u (1 + o(1))}$, where $u =
  \frac m b$, provided that $m^{1/100} \le b \le m^{99/100}$.
\end{theorem}

For the DLP in $G = \F_{p^n}^{\times}$, where~$p$ is fixed, we obtain
quite analogously a running time of $L(\frac 1 2, \sqrt 2)$.


\section{The number field sieve}\label{sec:nfs}

The number field sieve is an advanced index calculus method with
heuristic $L(1/3)$-complexity.  It was originally devised for the
integer factorisation problem, but the method can be adapted to apply
for the DLP in prime fields~\cite{Gor93,Sch93} and more generally
fields of large or medium characteristic~\cite{Sch00,JLSV06}.  The
principle difference between these algorithms and the basic version of
the previous section is that the two sides of each relation are of a
considerably smaller ``size'' than before, so that the smoothness
probability is -- at least heuristically -- greatly increased.

The setup for computing discrete logarithms in $\F_p^{\times}$ for a
prime~$p$ is as follows.  Let $f_1, f_2 \in \Z[X]$ be coprime
irreducible polynomials and $m \in \Z$ such that $f_1(m) \equiv f_2(m)
\equiv 0 \bmod p$.  To simplify the description we assume that~$f_1$
and~$f_2$ are monic, although this requirement is not essential.  We
have the following commutative diagram: \vspace{-2mm}
\begin{center}
  ~\xymatrix@R=8mm@C=2mm { & \Z[X] \ar[dl] \ar[dr] & \\
    \Z[X] / (f_1) \ar[dr]_{\ol{\ev}_m} & & 
    \Z[X] / (f_2) \ar[dl]^{\ol{\ev}_m} \\
    & \F_p & }~
\end{center}
Here, $\ol{\ev}_m \colon \Z[X] / (f_i) \to \F_p$ denotes the map $[ h ]
\mapsto [ h(m) ]$, for $i = 1, 2$.  For applying the index calculus
method we need a way to factor elements $[ h ]$ in the ring $\Z[X] /
(f_i)$ over a smoothness base, which is a more intricate issue,
requiring some concepts from algebraic number theory (see, e.g.,
\cite{Neu99}).

For $i = 1, 2$ let~$d_i$ be the degree and let~$x_i$ be a root of the
polynomial~$f_i$.  Then $K_i = \Q( x_i )$ is a number field, i.e., a
finite extension of~$\Q$, of degree~$d_i$, and its associated ring of
integers~$\mcO_i$ is a Dedekind domain, thus every nonzero ideal
factors uniquely into a product of prime ideals.  We have $\Z[X] /
(f_i) \cong \Z[ x_i ] \subseteq \mcO_i$ and the quotient group $\mcO_i
/ \Z[ x_i ]$ is finite.

In order to generate relations we make use of the multiplicative norm
map $N \colon K_i \to \Q$, which satisfies $N( \mcO_i ) \subseteq \Z$.
Associating to each nonzero ideal $\fa \subseteq \mcO_i$ its norm $N (
\fa ) = \# \mcO_i / \fa$, we have $| N(\alpha) | = N( (\alpha) )$ for
any nonzero element $\alpha \in \mcO_i$, and if $( \alpha ) = \prod
\fp^{e_{\fp}}$ is the decomposition of the principal ideal~$(\alpha)$
into prime ideals, then $|N( \alpha )| = \prod N(\fp)^{e_{\fp}}$.
Therefore, the prime ideal decomposition of $( \alpha )$ can be easily
obtained from the prime factorisation of the norm $N( \alpha ) \in
\Z$, provided that every prime ideal~$\fp$ containing $\alpha$ has
prime norm; this holds for example in the situation $\alpha = h(x_i)$
and $h = h_1 X + h_0$ with coprime integers $h_0, h_1 \in \Z$.

Accordingly, as factor base we choose elements, which are represented
by prime ideals in~$\mcO_i$ of small norm, and we look for linear
polynomials $h \in \Z[X]$ such that both norms $N(h(x_i)) = h_1^{d_i}
f_i( - \frac{h_0} {h_1} )$, $i = 1, 2$, are smooth.  After collecting
sufficiently many relations we obtain the ``virtual logarithms'' of
the prime ideals using linear algebra, cf.~\cite{JL03,Sch05}.  We
remark that one has to account for the unit groups of~$\mcO_1$
and~$\mcO_2$, for which one usually includes up to $(d_1 + d_2)$
further factor base elements corresponding to certain computable
logarithmic maps devised by Schirokauer~\cite{Sch93}.

\paragraph{Parameter choices and complexity analysis.}

Classically one chooses $d_2 = 1$, so that there is a rational side
$K_2 = \Q$ and $\mcO_2 = \Z$, as well as $x_2 = m \in \Z$.  Denote $f
= f_1$, $d = d_1$ and $x = x_1$.  Letting $h = h_1 X + h_0 \in \Z[X]$,
where $|h_i| \le E$, be the sieving polynomial one wants both $N( h(x)
)$ and $h(m)$ to be $B$-smooth.  To obtain an $L(1/3)$ algorithm one
sets $B = L(\smash{\frac 1 3}, \beta)$ and $E = L(\smash{\frac 1 3},
\eps)$, as well as $d = (\delta + o(1)) (\frac{\log N}{\log \log
  N})^{1/3}$ with parameters $\beta, \delta, \eps > 0$ to be
determined.

From the Kalkbrener bound~\cite{Kal97} one gets $\log |N( h(x))|
\approx \log \| f \|_{\infty} + d \log E$, where $\| f \|_{\infty}$
denotes the maximum absolute value of the coefficients of~$f$, while
$\log |h(m)| \approx \log m$.  One can choose $\| f \|_{\infty} \le m$
and $m \approx N^{1/d}$, so that $\log m \approx \frac 1 d \log N$,
and this implies $m = L(\frac 2 3, \frac 1 \delta)$.  Taking into
account $d \log E = \log L(\frac 2 3, \delta \eps)$ one thus obtains
\[ |N( h(x) ) h(m)| = L(\tfrac 2 3, \delta \eps 
+ \tfrac 2 \delta) . \] 
Under the heuristic assumption that the quantity $|N( h(x) )
h(m)|$ is uniformly distributed for a random polynomial~$h$, we get
from Corollary~\ref{cor:cop} for the probability~$P$ of this quantity
being $B$-smooth that $\frac 1 P = L(\frac 1 3, \smash{\frac{ \delta^2
    \eps + 2 } { 3 \beta \delta }} )$.

The sieving space size~$\approx E^2$ should be equal to the linear
algebra complexity~$\approx B^2$, therefore $\beta = \eps$, and as
about~$B$ relations are needed one sets $\frac 1 P \approx B$.
Therefore, we obtain the condition \smash{$\frac{ \delta^2 \beta + 2 }
  { 3 \beta \delta } = \beta$}, or $\delta^2 \beta + 2 = 3 \beta^2
\delta$.  The optimal choice $\delta^2 = 2 / \beta$ then yields
\smash{$\beta = ( \frac 8 9 )^{1/3}$}, resulting in a complexity of
\smash{$L(\frac 1 3, ( \frac{64} 9 )^{1/3} \approx 1.923 )$} for the
first two phases of the index calculus method.

For prime fields in which the prime is of particular shape the
so-called special number field sieve is available~\cite{Sem02}.  In
this case one may have small coefficients of~$f$, namely $\log \| f
\|_{\infty} = o(\log m)$, which leads to a faster algorithm.  Indeed,
one gets $\delta^2 \beta + 1 = 3 \beta^2 \delta$, and for $\delta^2 =
1 / \beta$ one obtains $\beta = ( \frac 4 9 )^{1/3}$, resulting in a
complexity of $L(\frac 1 3, ( \frac{32} 9 )^{1/3} \approx 1.526 )$.

Note that these two complexities are the same as for factoring using
the number field sieve and the special number field sieve,
respectively.

\paragraph{Individual logarithms.}

For obtaining an individual logarithm $\log_g h$ the previous
approach, i.e., multiplying~$h$ by random powers of~$g$ until it
factors over the factor base, is not viable as the factor base is now
much smaller.  Instead one uses a recursive strategy commonly referred
to as the \emph{descent}.  At each step of the descent one
\emph{eliminates} a given element which is not in the factor base,
i.e., rewrites it as a product of elements of smaller norm.  Starting
with~$h$, eliminations are carried out recursively until only factor
base elements remain.  In this way a tree is constructed with~$h$ as
root and with factor base elements as leaves, from which the logarithm
$\log_g h$ can easily be deduced from the factor base logarithms by
traversing the tree.


We omit the technical details for this section and simply remark that
this phase can be shown to have negligible (heuristic) complexity
compared to the other steps (see, e.g., \cite{CS06}).

\subsection{The medium number field sieve}

This variation of the number field sieve applies to the DLP in finite
fields of large or medium characteristic~\cite{JLSV06}.  

The setup is as follows.  Choose irreducible polynomials $f_1, f_2 \in
\Z[X]$ such that $f_1 \bmod p$ and $f_2 \bmod p$ have a common
irreducible factor $I \in \F_p[X]$ of degree~$n$.  (A simple choice is
to let $f_2 = f_1 + p$, if $f_1 \bmod p$ contains an irreducible
degree~$n$ factor; more advanced selection methods are sketched
below.)  We have thus the following commutative diagram:
\begin{center}
  ~\xymatrix@R=8mm@C=2mm { & \Z[X] \ar[dl] \ar[dr] & \\
    \Z[X] / ( {f_1} ) \ar[dr] & & \ar[dl] \Z[X] / ( {f_2} ) \\
    & \hspace{-6mm} \F_{p^n} \cong \F_p[X] / ( I )  \hspace{-6mm} & }~
\end{center}

As in the number field sieve for prime fields one obtains relations by
finding polynomials $h \in \Z[X]$ (of some degree $\le t$ and $\| h
\|_{\infty} \le E$) such that both norms $N(h(x_i)) = \Res(h, f_i)$
are $B$-smooth, where~$x_i$ is a root of~$f_i$ and $\Res$ denotes the
resultant.  The Kalkbrener bound~\cite{Kal97} implies here that
\[ \log | N(h(x_1)) N(h(x_2)) | \,\approx\, t (\log \| f_1
\|_{\infty} + \log \| f_2 \|_{\infty} ) + (d_1 + d_2) \log E , \]
where $d_i$ is the degree of~$f_i$ (if~$d_i$ is small enough).
Therefore, the running time
will crucially depend on the degree and the coefficient size of the
selected polynomials~$f_i$ in the setup phase of the algorithm.  The
recent work~\cite{BGGM15} achieves improvements by a clever polynomial
selection.

Indeed, for large characteristic one first chooses a polynomial $f_1
\in \Z[X]$ of degree $d+1$ with $\| f_1 \|_{\infty}$ small such that
$f_1 \bmod p$ has an irreducible factor~$I$ of degree~$n$.  Then a
polynomial $f_2 \in \Z[X]$ of degree~$d$ is chosen such that $I \mid
f_2 \bmod p$ and $\| f_2 \|_{\infty}$ as small as possible; this can
be achieved by lattice reduction~\cite{LLL82}, resulting in the
estimation $\log \| f_2 \|_{\infty} \approx \frac n {d+1} \log p$,
while one has $d_1 = d + 1$ and $d_2 = d \ge n$.  With suitably chosen
parameters this results in a running time of $L(\frac 1 3, (\frac {64}
9)^{1/3})$, the same as for the prime field case.

For medium characteristic the so-called Conjugation Method improves
upon the original selection method.  Here, one lets $\mu = Y^2 + a Y
+ b \in \Z[Y]$ be an irreducible polynomial with small coefficients
such that $\mu \bmod p$ has a root $\lambda \in \F_p$.  Then one
chooses $g_0, g_1 \in \Z[X]$ with $\| g_i \|_{\infty}$ small and $\deg
g_0 < n = \deg g_1$.  With this let $I = \lambda g_0 + g_1$, $f_2
= u g_0 + v g_1$, where $\lambda = \frac u v \bmod p$ with $\log \|
f_2 \|_{\infty} \approx \frac 1 2 \log p$, as well as $f_1 =
\Res_Y(\mu, Y g_0 - g_1) = g_1^2 + a g_1 g_0 + b g_0^2$, so that $\|
f_1 \|_{\infty}$ is small, while $d_1 = 2 n$ and $d_2 = n$.  The
complexity analysis results in a running time of $L(\frac 1 3, (\frac
{96} 9)^{1/3} \approx 2.201)$.

\paragraph{The tower number field sieve.}

An alternative method to deal with the DLP in extension fields stems
from work of Schirokauer~\cite{Sch00}, in which an algorithm for
finite fields~$\F_{p^n}$ of fixed extension degree~$n$ is presented
with a heuristic running time of $L(\frac 1 3, ( \frac{ 64 } 9 )^{1/3}
)$.  It uses a field extension on top of a number field~$F$, for which
its ring of integers~$\mcO_F$ satisfies $\mcO_F / (p) \cong \F_{p^n}$.
The setup is then very similar to the number field sieve for prime
fields, with~$\Z$ replaced by the ring~$\mcO_F$.

The algorithm has been analysed and extended in a recent
work~\cite{BGK15}, in which the authors coined the term tower number
field sieve.  In particular, it has been shown that it applies to the
whole range of large characteristic fields, achieving the same
complexity $L(\frac 1 3, ( \frac{64} 9)^{1/3} )$ as the medium number
field sieve of~\cite{JLSV06}, while being practically advantageous for
certain cases.



\section{The function field sieve -- classical variant}%
\label{sec:ffs}

The first (heuristic) $L(1/3)$ algorithm for the DLP in finite fields
was devised by Coppersmith in 1984~\cite{Cop84} -- thus predating the
number field sieve -- and as originally described applies to binary
finite fields, i.e., finite fields of characteristic~$2$.  As a result
of this algorithm, binary field DLPs fell out of favour in
cryptography; only much later with the advent of pairing-based
cryptography did they become fashionable once again, when the lower
security-per-bit they offer was considered tolerable thanks to the
functionality afforded by pairings on supersingular
curves~\cite{Gal01}.

The function field sieve was conceived of by Adleman~\cite{Adl94} in
analogy with the number field sieve, and later refined by Adleman and
Huang~\cite{AH99}, but can also be seen as a generalisation of
Coppersmith's algorithm. Joux and Lercier later provided a more
streamlined version~\cite{JL02}, as well as an improved version for
medium-size base fields~\cite{JL06}. In this section we briefly
describe these algorithms in turn, as not only are they historically
relevant to the DLP in small characteristic fields, but they also
serve as the benchmark against which the recent developments may be
compared.

Before describing the algorithms we remark that it is possible to map
between any two representations of a finite extension field
efficiently~\cite{Len91}; therefore when solving a DLP one is free to
choose a field-defining polynomial to one's advantage.


\subsection{Coppersmith's algorithm}

Coppersmith's algorithm~\cite{Cop84} is an index calculus algorithm
for solving the DLP in $\F_{2^n}$, represented as $\F_2[X] / ( I )$,
where $I \in \F_2[X]$ is an irreducible degree~$n$ polynomial of the
form $I = X^n + J$, where $J \in \F_2[X]$ is of relatively low degree
(less than $n^{2/3}$).  Let the factor base consist of the irreducible
polynomials up to a degree bound~$\le b$, and choose positive
integers~$h$ and~$\ell$ such that $h 2^{\ell} \ge n$.  For relation
generation, consider $f = u X^h + v \in \F_2[X]$, with $u, v \in
\F_2[X]$ coprime polynomials of degree~$\le d$, where~$d$ is a sieving
parameter, and compute
\[ f^{2^{\ell}} = ( u X^h + v )^{2^{\ell}} = u^{2^{\ell}} X^{h
  2^{\ell}} + v^{2^{\ell}} \equiv u^{2^{\ell}} X^{h 2^{\ell} - n} J +
v^{2^{\ell}} = g \pmod I . \] 

A relation is found if the polynomials $f, g \in \F_2[X]$ on both
sides are $b$-smooth.  Note that the corresponding degrees, namely
$\deg f \le h + d$ and $\deg g \le r + 2^{\ell} d$ with $r = \deg
X^{h 2^{\ell} - n} J$, can be made rather small by suitably chosen
parameters.  Indeed, we let $d = (c + o(1)) n^{1/3} (\log n)^{2/3}$
and suppose that $2^{\ell} \approx \sqrt{ \frac n d }$, as well as $h
= \lceil \frac n {2^{\ell}} \rceil$.  Then $\deg f$ and $\deg g$ are
about $\sqrt{ n d } = (\sqrt c + o(1)) n^{2/3} (\log n)^{1/3}$.
Choosing $b = (c + o(1)) n^{1/3} (\log n)^{2/3}$, by applying an
analogue of Corollary~\ref{cor:cop}, we get for the probability~$P$ of
both polynomials being $b$-smooth that $\frac 1 P = L( \frac 1 3,
\smash{\frac 2 {3 \sqrt c}} )$.  In order to generate enough relations
we set $\smash{\frac 2 {3 \sqrt c}} = c$, so that $c = ( \frac 4 9 )^{1/3}$,
resulting in an overall complexity of $L( \frac 1 3, ( \frac {32}
9)^{1/3} )$ for relation generation.  This matches the linear algebra
complexity using sparse matrix techniques, while the individual
logarithm phase was shown to have lower complexity.

This analysis supposes that $\sqrt{ \frac n d }$ is close to some
power $2^{\ell}$, which cannot be fulfilled for all~$n$.  In general
one obtains a complexity between $L( \frac 1 3, ( \frac {32} 9)^{1/3}
)$ and $L( \frac 1 3, 4^{1/3} )$, with $(\frac {32} 9)^{1/3} \approx
1.526$ and $4^{1/3} \approx 1.587$.  Observe that the algorithm
exploits the basic identity $(u + v)^2 = u^2 + v^2$, which holds for
any polynomials $u, v \in \F_2[X]$, and can thus be adapted easily to
the case of fields of fixed characteristic~$p$, by using the identity
$(u + v)^p = u^p + v^p$ for $u, v \in \F_p[X]$, with a correspondingly
increased upper bound on the resulting complexity.


\subsection{The function field sieve}

The function field sieve is an adaption of the number field sieve and
applies to finite fields~$\F_{p^n}$ of small characteristic~$p$.  The
principle idea is to replace~$\Z[X]$ by $\F_p[X, Y]$, and one usually
chooses the polynomial~$f_2$ to be of degree~$1$.  More precisely,
define~$\F_{p^n}$ as $\F_p[X] / ( I )$, where the irreducible
degree~$n$ polynomial $I \in \F_p[X]$ divides $f_d m^d + f_{d-1}
m^{d-1} + \dots + f_1 m + f_0 = f(m)$, with an irreducible polynomial
$f(Y) = \sum f_i Y^i \in (\F_p[X])[Y]$ and $m \in \F_p[X]$ suitably
chosen.  Then $f \in \F_p[X,Y]$ defines an algebraic curve~$C$ and its
associated function field $\F_p(C) = \on{Quot} \big( \F_q[X,Y] / (f)
\big)$.  Since $f(m) \equiv 0 \bmod I$ we have the following
commutative diagram:
\begin{center}
  ~\xymatrix@R=8mm@C=2mm { & \F_p[X, Y] \ar[dl] \ar[dr]^{Y \mapsto m} & \\
    \hspace{-4mm} \F_p[X, Y] / (f) \hspace{-4mm} \ar[dr]_{\phi} &  
    & \F_p[X] \ar[dl] \\
    & \hspace{-3mm} \F_{p^n} \cong \F_p[X] / ( I ) \hspace{-3mm} & }~
\end{center}

The factor base consists of all irreducible polynomials in $\F_p[X]$
up to degree~$b$.  In order to generate relations we search for
polynomials $h = h_1 Y + h_0 \in \F_p[X,Y]$ such that both $h_1 m +
h_0$ and the norm of~$h$ from $\F_p(C)$ to~$\F_p(X)$ are $b$-smooth.
In this case, we can express the divisor of $h_1 Y + h_0$ as a sum of
places in $\F_p(C)$ over factor base elements, and relate the
corresponding so-called surrogates via the map~$\phi$ with the factors
of $h_1 m + h_0$.  Then we solve as usual the logarithms of the factor
base elements by linear algebra.

As already mentioned there are three variations of the function field
sieve, which differ by the choice of the polynomials~$I$, $f$ and~$m$.
In Adleman's original version~\cite{Adl94} the field-defining
polynomial is written as $I = m^d + f_{d-1} m^{d-1} + \dots + f_1 m +
f_0$ using the base-$m$ technique prevalent in the number field sieve.
The refined version by Adleman and Huang~\cite{AH99} improves upon
this approach by choosing $I = X^n + J$ with $J \in \F_p[X]$ of small
degree $\deg J < n^{2/3}$, $f = Y^d + X^{ed-n} J$ and $m = X^e$,
which results in a complexity of $L( \frac 1 3, (\frac {32}
9)^{1/3})$, the same as for the special number field sieve.  It is
worth noting that Coppersmith's algorithm can be seen as a particular
case of this variation.  Finally, the function field sieve version of
Joux and Lercier~\cite{JL02} achieves the same complexity, but has
some practical advantages.  The idea is to start by choosing~$f$ with
coefficients of low degree in~$X$ and then letting $I = m_2^d f(
\frac{m_1} {m_2} )$ be the field-defining polynomial, where $m =
\frac{m_1} {m_2}$ is now a rational function.  In all cases the
individual logarithm phase can be shown to have a lower complexity
than the other two phases.


\subsection{The medium function field sieve}\label{sec:JL06}

Joux and Lercier in 2006 proposed the following simplified variant of
the function field sieve, which employs just the rational function
field of a univariate polynomial ring~\cite{JL06}.  The algorithm
applies to the whole range of finite fields $\F_{p^n}$ of medium-small
characteristic, i.e., $p = L_{p^n}(\alpha)$, where $\alpha \le \frac 1
3$. It also applies to extension fields $\F_{q^m}$, where~$q$ is any
prime power.

The representation of the field $\F_{q^m}$ is as follows.  Let $f, g
\in \F_q[X]$ be polynomials such that $g(f(X)) - X$ has an irreducible
factor $I \in \F_q[X]$ of degree~$m$.  Let~$x$ be a root of~$I$
in~$\F_{q^m}$ and let $y = f(x)$, hence $x = g(y)$.  Again we have the
following commutative diagram:
\begin{center}
  ~\xymatrix@R=8mm@C=4mm { & \!\!\F_q[X,Y]\!\! \ar[dl]_{Y \mapsto f(X)}
    \ar[dr]^{X \mapsto g(Y)} & \\
    \F_q[X] \ar[dr]_{\ev_x} & &
    \F_q[Y] \ar[dl]^{\ev_y} \\
    & \F_{q^m} & }~
\end{center}

Now if $q = L_{q^m}(1/3)$ then for $a, b, c \in \F_q$ we consider $h
= X Y + a Y + b X + c \in \F_q[X, Y]$, which leads in~$\F_{q^m}$ to
the following identity
\[ x f(x) + a f(x) + b x + c = g(y) y + a y + b g(y) + c . \] 
Let the factor base consist of the linear polynomials in~$X$ or in~$Y$.
If the corresponding polynomials on both sides, namely $h(X, f(X)) = X
f(X) + a f(X) + b X + c$ and $h(g(Y), Y) = g(Y) Y + a Y + b g(Y) + c$
are $1$-smooth, then a relation has been found.  We may choose the
polynomials~$f$ and~$g$ such that $\deg f, \deg g \approx \sqrt m$,
which leads to an algorithm with complexity $L( \frac 1 3, 3^{1/3}
\approx 1.442)$.  Here, the individual logarithm phase turns out to
have the same complexity as the main phase, with the bottleneck being
the elimination of degree two polynomials, see~\cite[Sec.~3]{JL06}.

In the general case, where $q = L_{q^m}(\alpha)$ with $\alpha \le 1 /
3$, to obtain an $L(1/3)$ algorithm we set as the degree bound for the
factor base $(C + o(1)) (\frac {\log q} {\log \log q})^{1/3 - \alpha}$
and consider polynomials of the form $h = h_1(X) Y + h_0(X)$ in order
to generate relations.  Note that for $\alpha = 0$ the case $\log q =
o(\log \log q^m)$ has to be treated extra with a slightly modified
analysis.  If $q = L_{q^m}(0)$, i.e., the case of small
characteristic, this results in an algorithm of complexity $L( \frac 1
3, (\frac {32} 9)^{1/3})$, with a less costly individual logarithm
phase.


\section{Small characteristic quasi-polynomial time algorithms}\label{sec:new}

In this section we present the two recent quasi-polynomial time algorithms for the small characteristic DLP.  The first of 
these was due to Barbulescu, Gaudry, Joux and Thom\'e~\cite{BGJT14}, while the second was due to Granger, Kleinjung and 
Zumbr\"agel~\cite{GKZ14b,GKZ15}.  For the purpose of assisting the reader in comparing the key insights and ideas behind the 
two algorithms, we present these for relation generation and individual logarithms, in turn.

\subsection{Resisting smoothness heuristics}

Before presenting the aforementioned developments, we start with some general remarks on how to obviate smoothness heuristics.

As detailed in Section~\ref{sec:icm}, a complexity of $L(1/2)$ could be said to be the ``natural complexity'' of the DLP when elements 
are generated uniformly in the multiplicative group. In order to obtain algorithms of complexity better than $L(1/2)$ -- at least for 
the relation generation phase of the index calculus method -- there are (at least) two approaches that one can explore. 
Firstly, one can try to generate relations between elements of smaller norm, which heuristically would have a higher probability 
of being sufficiently smooth, which is what the $L(1/3)$ algorithms do. Secondly, one can try to generate relations which have better than 
expected smoothness properties (or possibly a combination of both of these approaches). The second idea is perhaps far less obvious and 
more nuanced than the first; indeed, until recently it does not seem to have been appreciated that it was even a possibility, most likely 
because from an algorithm analysis perspective it is desirable that the expected smoothness properties hold. 


However, in 2013 a series of algorithmic breakthroughs occurred which demonstrated that for small characteristic 
fields the DLP is, at least heuristically, far easier than originally believed.
Central to these developments was the fundamental insight that one can produce families of polynomials 
that are smooth by construction~\cite{GGMZ13,Jou14} and use these to generate relations, 
while uniformly generated polynomials of the same degree are smooth with only an exponentially 
small probability. These results were the first to succeed in applying the second approach described above, and the 
families of polynomials used lay the foundation for the two quasi-polynomial algorithms.

\subsection{Polynomial time relation generation}

We now explain the polynomial time relation generation methods due to G\"olo\u{g}lu, Granger, McGuire and Zumbr\"agel~\cite{GGMZ13} and
Joux~\cite{Jou14}, which while different are essentially isomorphic, and were discovered independently at essentially the same time.

The finite fields to which the methods apply are of the form $\F_Q = \F_{q^{kn}}$, with $k \ge 2$ and $n \approx q$.  
A given small characteristic finite field $\F_{p^n}$ can be embedded into one of this form by letting $q = p^{\ell}$ 
with $\ell = \lceil \log_p{n} \rceil$, thus increasing the extension degree by a factor of $k \lceil \log_p{n} \rceil$, 
which does not impact upon the resulting complexity too much, see Section~\ref{ssec:qpas}.

\paragraph{The GGMZ method.}

The field setup used in~\cite{GGMZ13} can be seen in the context of the Joux-Lercier function
field sieve~\cite{JL06}, described in Section~\ref{sec:JL06}, but for which the degrees of the
polynomials~$f$ and~$g$ are extremely unbalanced.  In fact, we
consider $f = X^q$ and $g = \smash{\frac{h_0}{h_1}}$ for some $h_0,
h_1 \in \F_{q^k}[X]$ of low degree\footnote{In~\cite{GGMZ13} $h_1$ was not specified (and was thus implicitly $1$);
in~\cite{GKZ14a} $h_1$ was introduced in order to exploit other more convenient field representations.}, which leads to the following
field representation. We define $\F_{q^{kn}}$ as
$\F_{q^k}[X] / ( I )$, where $I \in \F_{q^k}[X]$ is an
irreducible degree~$n$ polynomial dividing $h_1(X^q) X - h_0(X^q)$ for
$h_0, h_1 \in \F_{q^k}[X]$ of low degree $\le d_h$.  Letting~$x$ be a root
of~$I$ in $\F_Q$, we have $y = f(x) = x^q$ and $x = g(y) =
\smash{\frac{h_0(y)}{h_1(y)}}$. The factor base consists of $h_1(x^q)$ and all linear polynomials on the $x$-side; note that the $y$-side 
factor base is not needed since for all $d \in \F_{q^k}$ one has $(y + d) = (x + d^{1/q})^q$. 

As in Section~\ref{sec:JL06}, let $a,b,c$ be in the base field $\F_{q^k}$, and consider elements of $\F_{q^{kn}}$ of the form 
$x y + a y + b x + c$. Using the field isomorphisms we have the following identity for such elements:
\begin{equation}\label{basicidentity}
  x^{q+1} + a x^q + b x + c = \tfrac 1 {h_1(y)} \big( y h_0(y) + a y
  h_1(y) + b h_0(y) + c h_1(y) \big) \,. 
\end{equation}
An extremely useful observation is that the l.h.s.\ of Eq.~(\ref{basicidentity}) splits completely over $\F_{q^k}$ with probability 
$\approx 1/q^3$, which is exponentially higher than the splitting probability of a uniformly random polynomial of the same degree, which 
is $\approx 1/(q+1)!$. Indeed, let $k \ge 3$ and consider the polynomial $X^{q+1} + aX^{q} + bX + c$. If $ ab\ne c$ and $a^{q}\ne b$, this 
may be transformed (up to a scalar) into
\[ F_B(\overline{X}) = \overline{X}^{q+1} + B \overline{X} + B \:, \quad \text{with} \quad B = \frac{(b - a^{q})^{q+1}} {(c - ab)^{q}} \:, \] 
via $\smash{X = \frac{c - ab}{b - a^q} \, \overline{X} - a}$. Clearly, the original polynomial splits whenever~$F_B$ splits and we have a
valid transformation from $\overline{X}$ to $X$. The following theorem of Bluher counts the number of $B \in \F_{q^k}$ for which~$\F_B$ 
splits completely over~$\F_{q^k}$.$\!\!$
\begin{theorem}\label{thm:bluher} \emph{(Bluher~\cite{Blu04})}
  The number of elements $B \in \F_{q^k}^{\times}$ s.t. the polynomial $\smash{F_B(\overline{X}) \in \F_{q^k}[\overline{X}]}$ 
  splits completely over $\F_{q^k}$ equals
  \[ \frac{q^{k-1}-1}{q^{2}-1} \quad \text{if } k \text{ is odd} \:, \qquad
  \frac{q^{k-1} - q}{q^{2}-1} \quad \text{if } k \text{ is even} \:. \]
\end{theorem}
Using the expression for $B$ as a function of $a,b,c$ one can easily generate triples $(a,b,c)$ for which the l.h.s.\ \emph{always} splits.
In particular, one first computes the set $\mathcal{B}$ of all $B \in \F_{q^k}$ for which $F_B$ splits over $\F_{q^k}$. Then for any 
$a, b \ne a^q$ and $B \in \mathcal{B}$ there is a unique $c$ for which the l.h.s.\ splits. Furthermore, by Theorem~\ref{thm:bluher} there 
are $\approx q^{3k - 3}$ such triples, and for each one whenever the r.h.s.\ splits as well, one has a relation. For $k = 2$ there are no 
such $F_B$, however in this case the set of triples for which the l.h.s.\ splits non-trivially is easily shown to be
\[ \{ (a, a^q, c) \mid a \in \F_{q^2} \text{ and } c \in \F_q, c \ne a^{q+1} \} . \]
In all cases, assuming the r.h.s.\ splits with probability $1/(d_h+1)!$, in order for there to be sufficiently many relations one requires 
that $q^{2k-3} > (d_h + 1)!$. For fixed~$d_h$ and $q \rightarrow \infty$ the cost of computing the logarithms of all the factor base 
elements is heuristically $O(q^{2k + 1})$ (operations in $\Z / (Q-1) \Z$), using sparse (weight~$q$) linear algebra techniques, which is polynomial in 
$\log Q = q^{1 + o(1)}$ for fixed~$k$.

\paragraph{Joux's method.}

The method of Joux~\cite{Jou14} works for fields of the same shape as those in the GGMZ method (although only $k=2$ was used for the
exposition and initial examples), however the field representation is slightly different. In particular, let $\F_{Q} = \F_{q^k}(x) 
= \F_{q^k}[X] / ( I )$, where $I \mid h_1(X) X^q - h_0(X)$ for some $h_0, h_1 \in \F_{q^k}[X]$ of low degree $\le d_h$, and hence 
$x^q = \smash{\frac{h_0(x)} {h_1(x)}}$. Fixing some notation, for any $a \in \F_{q^k}[X]$ we have $a(X)^q = \tilde a(X^q)$ with $\deg \tilde a = \deg a$, 
where the coefficients of~$\tilde a$ are the $q$-th powers of the coefficients of~$a$.

As already mentioned, Joux's method is essentially isomorphic to the GGMZ method and starts with the identity 
\[ \smprod_{\mu \in \F_q} (X - \mu) = X^q - X . \]
Substituting~$X$ by $\smash{\frac{\alpha X + \beta} {\gamma X + \delta}}$ with 
$\alpha, \beta, \gamma, \delta \in \F_{q^k}$ and $\alpha \delta - \beta \gamma \ne 0$,
and multiplying by $(\gamma X + \delta)^{q+1}$ one obtains
\begin{equation}\label{jouxidentity}
  (\gamma X \!+\! \delta) \!\smprod_{\mu \in \F_q}\! \!\big( (\alpha X \!+\! \beta) - \mu (\gamma X \!+\! \delta) \big)\! 
  = (\alpha X \!+\! \beta)^q (\gamma X \!+\! \delta) - (\alpha X \!+\! \beta) (\gamma X \!+\! \delta)^q .
\end{equation}
Observe that the r.h.s.\ of Eq.~(\ref{jouxidentity}) has (up to a scalar factor) the same form as the l.h.s.\ of Eq.~(\ref{basicidentity}), 
and automatically splits over $\F_{q^k}$ by virtue of the l.h.s.\ of Eq.~(\ref{jouxidentity}).
Using the field equation $x^q = \smash{\frac{h_0(x)} {h_1(x)}}$, the r.h.s.\ of Eq.~(\ref{jouxidentity}) becomes
\[ \tfrac{1}{h_1(x)} (\tilde{\alpha} h_0(x) + \tilde{\beta} h_1(x)) (\gamma x + \delta) - 
(\alpha x + \beta) (\tilde{\gamma} h_0(x) + \tilde{\delta} h_1(x)), \]
and if it also splits over $\F_{q^k}$ then one has a relation between the linear elements and $h_1(x)$. As for the number of different 
relations one can obtain in this way, observe that the total number of $(\alpha, \beta, \gamma, \delta)$-transformations is just 
$|\on{PGL}_2(\F_{q^k})| = q^{3k} - q^k$.
However, two transformations will give the same relation (up to multiplication by a scalar in $\F_{q^k}^{\times}$) if there exists an element
of $\on{PGL}_2(\F_q)$ which gives the second transformation when multiplied by the first. Hence the total number of different transformations
is $\approx q^{3k-3}$, just as for the GGMZ method. One should therefore compute a set of coset representatives for the quotient
$\on{PGL}_2(\F_{q^k})/\on{PGL}_2(\F_q)$ to avoid repetitions; the GGMZ method by contrast already does this implicitly.

\subsection{Degree two elimination}\label{sec:individual}

The methods of the previous subsection demonstrate that the factor base can be extremely small and the logarithms of its elements computed 
efficiently. Therefore, our attention now falls on the individual logarithm phase. For the $L(1/2)$ and $L(1/3)$ algorithms this phase had those 
respective complexities, although usually with smaller constants. In contrast, for the new algorithms it becomes the dominant phase.



If one employs the usual descent method for the individual logarithm phase for our setup, then as in~\cite{JL06} the degree two eliminations are 
the bottleneck.  However, the work of GGMZ and Joux shows how to eliminate degree two elements efficiently, and the respective methods 
were developed into the building blocks of the quasi-polynomial descent algorithms.  Thus, due to their importance, we now present them.

\paragraph{The GGMZ degree two elimination method.}

As before consider $\F_{q^{kn}} = \F_{q^k}(x)$ with $y = x^q$ and $x = \smash{\frac{h_0(y)} {h_1(y)}}$.  Let $P(x)$ be a degree two irreducible 
element to be eliminated, i.e., written as a product of linear elements, and let $\tilde{P}(y) = P(x)^q$.
For $a, b, c \in \smash{\F_{q^k}}$ consider Eq.~(\ref{basicidentity}) from the setup for relation generation.
Imposing the condition that $\tilde{P}(y)$ divides the r.h.s.\ of Eq.~(\ref{basicidentity}), we get an $\F_{q^k}$-linear system 
in $(a, b, c)$, whose solution can in general be expressed in terms of~$a$.  In particular, there exists
$u_0, u_1, v_0, v_1 \in \F_{q^k}$ dependent on~$P(x)$ such that $b = u_0 + a u_1$ and $c = v_0 + a v_1$ and hence
\[ x^{q+1} + a x^q + b x + c = \tfrac 1 {h_1(y)} \big( h_0 y + a h_1 y + (u_0 + a u_1) h_0 + (v_0 + a v_1) h_1 \big) . \]
The r.h.s.\ has degree $d_h + 1$, so the cofactor heuristically splits with probability $\approx 1/(d_h-1)!$, while the l.h.s.\ 
heuristically splits with probability $1/q^3$ for randomly chosen~$a$. 

To find such~$a$ more directly, we again use the set $\mathcal{B}$ of splitting $F_B$ polynomials. In particular, as a function of~$a$ we 
transform the l.h.s.\ to obtain an~$F_B$, so that 
\begin{equation}\label{bs}
  B = \frac{( -a^q + u_0 + a u_1)^{q+1}} {(-u_1 a^2 + (v_1 - u_0) a + v_0)^q} .
\end{equation}
Then for each $B \in \mathcal{B}$ we solve Eq.~(\ref{bs}) for~$a$, which one can do by taking the greatest common divisor with $a^{q^k} - a$
for instance, which takes time which is polynomial in $\log Q$.

\paragraph{Joux's degree two elimination method.}

In~\cite{Jou14} Joux used the following extension of the degree one relation generation to compute the logarithms of the degree two elements
in batches. In particular, let $u \in \F_{q^k}$ and substitute~$X$ in~$X^q - X$ by \[ \frac{\alpha (X^2 + u X) + \beta} {\gamma (X^2 + u X)
 + \delta} , \] with $\alpha, \beta, \gamma, \delta \in \F_{q^k}$ and $\alpha \delta - \beta \gamma \ne 0$, and multiply by 
$(\gamma (X^2 + u X) + \delta)^{q+1}$. This gives
\begin{multline}\label{jouxidentity2}
  (\gamma (X^2 \!+ u X) + \delta) \smprod_{\mu \in \F_q} \big( (\alpha(X^2 \!+ u X) + \beta) - \mu (\gamma (X^2 \!+ u X) + \delta) \big) \\
  = (\alpha (X^2 \!\!+\! u X) \!+\! \beta)^q (\gamma (X^2 \!\!+\! u X) \!+\! \delta) \!-\! (\alpha (X^2 \!\!+\! u X) \!+\! \beta) 
  (\gamma (X^2 \!\!+\! u X) \!+\! \delta)^q . \!\!
\end{multline}
Observe that in the l.h.s.\ of Eq.~(\ref{jouxidentity2}) all of the factors are of the form $X^2 + uX + v$ for some $v \in \F_{q^k}$. Replacing 
all occurrences of $X^q$ by $\smash{\frac{h_0(X)} {h_1(X)}}$ on the r.h.s.\ of Eq.~(\ref{jouxidentity2}) gives a low degree expression, which 
if $1$-smooth gives a relation. Once more than $\approx q^k/2$ such relations have been obtained -- which is the expected number of irreducibles 
of this form -- one may take logs and solve the resulting system. Doing this for each $u \in \smash{\F_{q^k}}$ means the logarithms of all degree two 
elements can be computed in time $O(q^{3k+1})$ .

Joux also gave a new elimination method which relies on Gr\"obner basis computations, whose cost increases with the degree.
Balancing the costs of the Gr\"obner basis descent and the classical descent (whose cost decreases with the degree) results in a 
heuristic $L(1/4 + o(1))$ algorithm, which was the first algorithm to break the long-standing $L(1/3)$ barrier.



\subsection{Quasi-polynomial time descent}\label{ssec:qpas}

In this subsection we describe the quasi-polynomial descent algorithms
of BGJT~\cite{BGJT14} and GKZ~\cite{GKZ14b,GKZ15}.

\paragraph{The BGJT quasi-polynomial algorithm.}

Joux's idea for performing degree two elimination in batches can be
generalised to polynomials~$P$ of any degree, which leads to the quasi-%
polynomial descent algorithm of Barbulescu \emph{et al}.~\cite{BGJT14}.  
As before, let the finite field $\F_Q = \F_{q^{kn}}$ be given as
$\F_{q^k}[X] / ( I )$, where the degree~$n$ irreducible polynomial $I
\in \F_{q^k}[X]$ divides $h_1(X) X^q - h_0(X)$ for polynomials $h_0,
h_1 \in \F_{q^k}[X]$ of low degree $\le d_h$.  Suppose an element to
be eliminated is represented by a polynomial $P \in \F_{q^k}[X]$ of
some degree $D < n$.  The goal of each elimination step is to find an
expression for~$P$ in terms of elements of degree less than $D/2$.
Substituting $X$ by~$P$ in Eq.~(\ref{jouxidentity}) one obtains
\[ (\gamma P \!+\! \delta) \!\smprod_{\mu \in \F_q}\! \big( (\alpha P
\!+\!  \beta) \!-\! \mu (\gamma P \!+\! \delta) \big) = (\alpha P
\!+\! \beta)^q (\gamma P \!+\! \delta) \!-\! (\alpha P \!+\! \beta)
(\gamma P \!+\!  \delta)^q , \] for any $\alpha, \beta, \gamma, \delta
\in \F_{q^k}$ with $\alpha \delta - \beta \gamma \ne 0$.  Observe that
all factors in the l.h.s.\ are (up to a scalar) of the form $P + v$
for some $v \in \F_{q^k}$.  Similarly to before, replacing all
occurrences of $X^q$ by $\smash{\frac{h_0(X)} {h_1(X)}}$ on the
r.h.s. gives an expression of relatively low degree $(d_h + 1) D$,
which if $D/2$-smooth is viewed as a relation.  In order to get a
degree $D/2$-smooth expression for~$P$ not involving its
$\F_{q^k}$-translates, one collects more than $\approx q^k$ relations
of the above form, so that the associated linear system (in the
logarithms of the factors $P + v$, for $v \in \F_{q^k}$) is expected
to have full rank, and performs a linear algebra elimination to find
an expression for~$P$ as a $D/2$-smooth product, as desired.

By this process the polynomial~$P$ is rewritten modulo~$I$ as a
product $\prod_i R_i^{e_i}$ of polynomials of degree $\le D / 2$, in
time polynomial in~$q^k$ and~$n$, where the number of factors~$R_i$ is
in~$O(n q^k)$.  Iterating these elimination steps down to the factor
base gives for fixed~ $d_h$ and~$k$ a heuristic descent running time of
\begin{equation}\label{complexity} 
  q^{O(\log n)} = \exp( O(\log q \log n) ) ,
\end{equation}
which is, taking into account the polynomial time relation generation,
also the heuristic total complexity for solving the DLP in $\F_Q$.

In fact, the running time of the algorithm can be improved slightly by
replacing the degree bound $D / 2$ above by $O( D \log \log q / \log q
)$, leading to the heuristic complexity $q^{O(\log n / \log \log q)}$.
While this quasi-polynomial algorithm is asymptotically the fastest,
it has however not yet been used in record computations, see
Section~\ref{ssec:practical}.

\paragraph{Impact on small and medium-small characteristic.}

Having a quasi-polynomial time algorithm for finite fields of the form
$\F_Q = \F_{q^{kn}}$ affects also the DLP in general finite fields of
small or medium-small characteristic.  Indeed, if a DLP in a finite
field $\F_{p^n}$ is considered, we embed this field in $\F_Q$, where
$Q = q^{kn} = (p^n)^{k{\ell}}$ and $q = p^{\ell}$.

In the case of small characteristic, i.e., $p = L_{p^n}(0)$, we have
$\log p = O(\log n)$ and $\log \log p^n = \log n + \log \log p \approx
\log n$.  We let $q = p^{\ell}$ such that $q \ge n - d_h$ and $\log q
= O(\log n)$, hence Eq.~(\ref{complexity}) implies a running time of
$\exp( O(\log \log p^n)^2))$, which is quasi-polynomial in the
logarithm $\log p^n$ of the group order.

If the characteristic is medium-small, i.e., $p = L_{p^n}(\alpha)$,
where $0 < \alpha < 1/3$, we let $q = p$, so that $\log q = \log p =
O( (\log p^n)^{\alpha} (\log \log p^n)^{1-\alpha} )$.  By Eq.~(\ref{complexity})
and observing that $\log n = O(\log \log p^n)$ we get a complexity of
$L_{p^n}(\alpha) ^{O(\log \log p^n)} = \exp( O( (\log p^n)^{\alpha}
(\log \log p^n)^{2 - \alpha}) = L_{p^n}(\alpha + o(1))$, which
improves on the function field sieve having $L(1/3)$-complexity.

\paragraph{Eliminating smoothness heuristics: The GKZ quasi-polynomial algorithm.}

The prototype algorithm~\cite{GKZ14b} and its rigorous analysis~\cite{GKZ15} uses a novel descent strategy, which was developed from 
the GGMZ degree two elimination technique and is thus independent of Joux's degree two elimination and the BGJT algorithm.
The algorithm has quasi-polynomial complexity for the DLP in small characteristic (as well the same complexity as the BGJT algorithm
for $q = L_Q(\alpha)$), and it features a rigorous complexity analysis, which is based on the absolute irreducibility of a certain algebraic 
curve, once an appropriate field representation is found. A remarkable property of the descent method is that it does not require any 
smoothness assumptions about non-uniformly distributed polynomials, in contrast to all previous index calculus algorithms.
The new descent method also has practical advantages and has been used in record computations, see Section~\ref{ssec:practical}.

Consider in $\F_Q = \F_{q^{kn}}$ an element to be eliminated which is represented by an irreducible polynomial $P \in \F_{q^k}[X]$ of degree $2 d$, 
and observe that over the degree~$d$ extension $\F_{q^{kd}}$ this polynomial factors into~$d$ irreducible quadratics.  Applying the degree two
elimination method of GGMZ to any of these quadratics enables one to rewrite the quadratic as a product of linear elements over $\F_{q^{kd}}$.
Then one simply applies the norm map down to $\F_{q^k}$, which takes the linear elements to powers of irreducible elements of degree 
dividing~$d$ and the quadratic element back to the original element~$P$, thus completing its elimination.  
If the target element has degree a power of two then this elimination can be applied recursively, halving the degree (or more) of the 
elements in the descent tree upon each iteration.  For the purpose of building a full DLP algorithm which may be applied to any target 
element, one can use a Dirichlet-type theorem due to Wan~\cite[Thm.~5.1]{Wan97} to ensure that any field element is given by an 
irreducible of degree a power of two only slightly larger than~$n$.

To prove that the degree two elimination step represented in Eq.~(\ref{bs}) always succeeds, the following characterisation of the set of 
$B \in \mathcal{B}$ is used. In particular, by an elementary extension of~\cite[Theorem~5]{HK10} the set~$\mathcal{B}$ is 
the image of $\F_{q^{kd}} \setminus \F_{q^2}$ under the map
\[ z \mapsto \frac{(z - z^{q^2})^{q+1}} {(z - z^q)^{q^2+1}} . \]
By this and Eq.~(\ref{bs}), in order to eliminate~$P$ we need to find $(a, z) \in \F_{q^{kd}} \times (\F_{q^{kd}} \setminus \F_{q^2})$ satisfying
\[ (z - z^{q^2})^{q+1}(-u_1 a^2 \!+\! (v_1 \!-\! u_0) a \!+\! v_0)^q -  (z - z^q)^{q^2+1} (-a^q \!+\! u_0 \!+\! a u_1)^{q+1} = 0 . \]

That there are sufficiently many appropriate $(a, z)$ pairs was proven using the classification of subgroups of $\on{PGL}_2(\F_{q^{kd}})$
and the Weil bound for absolutely irreducible curves.  One also needs to deal carefully with so-called \emph{descent traps}, i.e.,
elements that divide $h_1(X) X^{q^{kd+1}} - h_0(X)$ for $d \ge 0$ which can be shown to be ineliminable in the above manner. 
These considerations led to the following theorem.

\begin{theorem}\label{thm-main}
  Given a prime power $q>61$ that is not a power of~$4$, an integer $k \ge 18$, coprime polynomials $h_0, h_1 \in \F_{q^k}[X]$ of degree at 
  most two and an irreducible degree~$n$ factor~$I$ of ${h_1 X^q - h_0}$, the DLP in $\F_{q^{kn}} \cong \F_{q^k}[X]/(I)$ can be solved in 
  expected time \[ q^{\log_2 n + O( k )} . \]
\end{theorem}
That the degree of $h_0, h_1$ is at most two is essential to eliminating smoothness heuristics, since this ensures that the cofactor of the 
r.h.s.\ of Eq.~(\ref{basicidentity}) has degree at most one, and is thus automatically $1$-smooth. 

Thanks to Kummer theory, such $h_1, h_0$ are known to exist when $n = q-1$, which gives the following easy corollary when $m = i k (p^i - 1)$.
\begin{theorem}\label{thm-sequence}
  For every prime~$p$ there exist infinitely many explicit extension fields $\F_{p^m}$ in which the DLP can be solved in expected 
  quasi-polynomial time 
  \[ \exp \big( ( 1 / \log 2 + o(1) ) (\log m)^2 \big) . \]
\end{theorem}
Proving the existence of $h_0, h_1$ for general extension degrees as per Theorem~\ref{thm-main} seems to be a hard problem, even if 
heuristically it would appear to be almost certain, and in practice it is very easy to find such polynomials. 

\subsection{Practical considerations}\label{ssec:practical}

While rigorously proving the correctness of a new DLP algorithm is of
theoretical interest, a perhaps more immediate measure of its value
arises from its practical impact.  Furthermore, the challenge of
setting new DLP records often leads to theoretical insights that give
rise to novel or superior algorithms, which may not have been at all
obvious.  In addition, such large-scale computations are of
cryptologic relevance since they allow one to assess the security of
contemporary or future DLP-based cryptosystems.  Hence, the value of
practical considerations should not be underestimated.

\paragraph{Kummer extensions and automorphisms.}

Kummer theory provides us with particularly useful polynomials for the
field representation, as observed in~\cite{Jou13}.  In fact, whenever
$n \mid q^k - 1$ there exists $c \in \F_{q^k}$ such that $X^n - c \in
\F_{q^k}[X]$ is irreducible (one may, for instance, take~$c$ to be a
multiplicative generator).  In particular, for $n = q-1$ we have $I =
X^{q-1} - c \mid X^q - c X = h_1 X^q - h_0$, with $h_1 = 1$, $h_0 = c
X$ of degree $\le d_h = 1$.

Having degree at most one for the polynomials $h_0, h_1$ in the field
representation has some practical advantages for the relation
generation and especially for the individual logarithm phase.
Furthermore, when defining $\F_{q^{kn}}$ by $\F_{q^k}[X] / ( I )$ one
can use factor base preserving automorphisms to reduce the complexity
of the linear algebra step.  In fact, for the $q$-th power Frobenius
one has $(x + a)^q = x^q + a^q = c x + a^q = c ( x + \frac{a^q} {c} )$.
The group generated by the Frobenius automorphism of order~$k n$ acts
on the factor base, effectively reducing the variables of the linear
algebra problem by a factor of about~$k n$.  Similar observations hold
for so-called twisted Kummer extensions, when $n \mid q + 1$.

\paragraph{New discrete logarithm records.}

\begin{table}[t]
  \caption{Discrete logarithm record computations in finite fields of small or 
    medium characteristic.
    Details, as well as further announcements, can be found in the number theory
    mailing list (\url{https://listserv.nodak.edu/cgi-bin/wa.exe?A0=NMBRTHRY}).%
    \vspace{-4mm}}
  \label{tab:records}

  \begin{center}
    \begin{tabular}{c|c|c|c|c}
      \,bitlength\, & \,charact.\, & \,Kummer\, & who/when & running time \\
      & & & & \\[-3.5mm] \hline & & & & \\[-3mm]
      127 & 2 & no & Coppersmith 1984~\cite{Cop84} & ~$L(1/3\,,\, 1.526..1.587)$~ \\
      401 & 2 & no & Gordon, McCurley 1992~\cite{GM93} & $L(1/3\,,\, 1.526..1.587)$ \\
      521 & 2 & no & Joux, Lercier 2001~\cite{JL02} & $L(1/3\,,\, 1.526)$ \\
      607 & 2 & no & Thom\'e 2002 & $L(1/3\,,\, 1.526..1.587)$ \\
      613 & 2 & no & Joux, Lercier 2005 & $L(1/3\,,\, 1.526)$ \\
      556 & \,medium\, & yes & Joux, Lercier 2006~\cite{JL06} & $L(1/3\,,\, 1.442)$ \\
      676 & 3 & no & Hayashi \emph{et al}. 2010~\cite{H+10} & $L(1/3\,,\, 1.442)$ \\
      923 & 3 & no & Hayashi \emph{et al}. 2012~\cite{HSST12} & $L(1/3\,,\, 1.442)$ \\
      1175 & medium & yes & Joux 24 Dec 2012~\cite{Jou13} & $L(1/3\,,\, 1.260)$ \\ 
      1425 & medium & yes & Joux 6 Jan 2013~\cite{Jou13} & $L(1/3\,,\, 1.260)$ \\
      1778 & 2 & yes & Joux 11 Feb 2013~\cite{Jou14} & $L(1/4 + o(1))$ \\
      1971 & 2 & yes & GGMZ 19 Feb 2013~\cite{GGMZ13}~ & $L(1/3\,,\, 0.763)$ \\
      4080 & 2 & yes & Joux 22 Mar 2013~\cite{Jou14} & $L(1/4 + o(1))$ \\
      6120 & 2 & yes & GGMZ 11 Apr 2013~\cite{GGMZ14} & $L(1/4)$ \\
      6168 & 2 & yes & Joux 21 May 2013 & $L(1/4 + o(1))$ \\
      1303 & 3 & no & AMOR 27 Jan 2014~\cite{AMOR14b} & $L(1/4 + o(1))$ \\ 
      4404 & 2 & no & GKZ 30 Jan 2014~\cite{GKZ14a} & $L(1/4 + o(1))$ \\ 
      9234 & 2 & yes & GKZ 31 Jan 2014 & $L(1/4 + o(1))$ \\
      3796 & 3 & no & \,Joux, Pierrot 15 Sep 2014~\cite{JP14}\, & $L(0 + o(1))$ \\ 
      1279 & 2 & no & Kleinjung 17 Oct 2014 &  $L(0 + o(1))$ \\
    \end{tabular}
  \end{center}
\end{table}

\enlargethispage{2mm}

In order to obtain the best performance in practice one typically uses
a combination of elimination techniques to develop the descent
tree. For currently considered bitlengths these consist of: 
\begin{itemize}
\item Continued fraction initial split,
\item Large degree classical descent,
\item quasi-polynomial time descent~\cite{GKZ14b,GKZ15},
\item low degree Gr\"obner Basis descent~\cite{Jou14},
\item degree two elimination~\cite{GGMZ13,Jou14}.
\end{itemize}

The new DLP algorithms have lead to a series of records, see
Table~\ref{tab:records}.  At the time of writing the largest one was
in the field of $2^{9234}$ elements.

\paragraph{Impact on pairing-based cryptography.}

The recent advances in solving the DLP in finite fields of small
characteristic have had a considerable impact on cryptology research,
particularly in the area of identity-based cryptography in which such
DLPs are extremely important. This application was investigated by Adj
\emph{et al.}~\cite{AMOR14a} and Granger~\emph{et al.}~\cite{GKZ14a}.
In particular, in~\cite{GKZ14a} the new methods were improved and
extended to make the algorithms more efficient and more widely
applicable. At the ``128-bit security level'' it was shown that a
common genus one curve offered only~$59$ bits of security (which is
now much reduced thanks to~\cite{GKZ14b}), while a genus two curve
over $\F_{2^{367}}$ with embedding degree $12$ was totally broken.



\section{Summary and outlook}\label{sec:conclusion}

After three decades of relatively little progress in the DLP in small
characteristic fields, in the past three years the area has
transformed dramatically.  Not only are the algorithms much faster,
but the ingredients are much simpler too, which raises the question of
why these techniques were not discovered earlier? While not historians
we proffer two tentative speculations.

Firstly, since Coppersmith's algorithm was so much faster than the
contemporary $L(1/2)$ algorithms, small characteristic fields were
avoided in cryptographic applications and so became less important to
study as attention moved away to larger characteristic fields (this
changed however with the advent of pairing-based cryptography).

Secondly, and perhaps crucially, one might deduce from
Theorems~\ref{thm:cop} and~\ref{thm:ol} that for large prime fields
and small characteristic fields, DLP algorithms that rely on
smoothness heuristics should have the same complexity when index
calculus methods are used.  As well as being true for the above two
algorithms, in which elements are generated uniformly, it is also the
case -- at least heuristically -- for the number field sieve and
function field sieve, in which elements are generated non-uniformly.
Hence algorithms which exploit the additional structure of small
characteristic fields, namely the splitting of $X^q - X$, were simply
not considered.

Looking ahead, there are three natural problems that remain open. The
first is to provide a rigorous quasi-polynomial algorithm for the
small characteristic DLP. The second, perhaps more challenging problem
is to find a polynomial time algorithm for small characteristic,
either heuristic or rigorous. Finally, the third, probably much harder
problem is to develop faster algorithms for medium and large
characteristic. Since there is less structure for prime fields in
particular, it seems that fundamentally new ideas will be required.


\end{document}